\documentclass[11pt]{article} % for dvi

\usepackage{graphicx}
\usepackage{amsfonts}
\usepackage{amsmath}
\usepackage{amscd}
\usepackage{amssymb}
\usepackage{verbatim}

\newenvironment{proof}{\noindent\emph{Proof}.}{$\square$\smallskip}
\newenvironment{proof*}{\noindent\emph{Proof}}{$\square$\smallskip}

\newtheorem{theorem}{Theorem}[section]

\newtheorem{Definition}[theorem]{Definition}
\newtheorem{lemma}[theorem]{Lemma}
\newtheorem{Example}[theorem]{Example}
\newtheorem{corollary}[theorem]{Corollary}
\newtheorem{Remark}[theorem]{Remark}
\newtheorem{proposition}[theorem]{Proposition}
\newtheorem{conjecture}[theorem]{Conjecture}

\newtheorem{Exercise}[theorem]{Exercise}
\newtheorem{Exercises}[theorem]{Exercises}
\newtheorem{Notation}[theorem]{Notation}

\newenvironment{definition}{\begin{Definition}\normalfont}{\end{Definition}}
\newenvironment{example}{\begin{Example}\normalfont}{\end{Example}}
\newenvironment{remark}{\begin{Remark}\normalfont}{\end{Remark}}

\newcommand{\br}{\ensuremath{\mathbb{R}}} %The real numbers
\newcommand{\bz}{\ensuremath{\mathbb{Z}}} %The integers
\newcommand{\bn}{\ensuremath{\mathbb{N}}} %The natural numbers
\newcommand{\bc}{\ensuremath{\mathbb{C}}} %The complex numbers
\newcommand{\id}{\ensuremath{\mathrm{id}}} %The identity morphism
\newcommand{\incl}{\ensuremath{\mathrm{incl}}} %The inclusion morphism
\newcommand{\co}{\ensuremath{\colon}} % colon for funnctions
\newcommand{\sm}{\ensuremath{{\rm Sim}}} %Notation for similarity structure
\newcommand{\Aut}{\ensuremath{{\rm Aut}}} %Notation for automorphism group
\newcommand{\bi}{\ensuremath{{\rm Bi}}} %Notation for the set of bijections
                            
\title{Local Similarities and the Haagerup Property}  
\author{Bruce Hughes (Vanderbilt University)\\
\\
with an appendix by Daniel S. Farley (Miami University)}
\begin{document}

\maketitle

\begin{abstract}A new class of groups, the locally finitely determined groups of local similarities on compact ultrametric spaces, is introduced and it is proved that these groups have the Haagerup property (that is, they are a-T-menable in the sense of Gromov).
The class includes Thompson's groups, which have already
been shown to have the Haagerup property by D. S. Farley, as
well as many other groups acting on boundaries of trees.
A sufficient condition, used in this paper, for the Haagerup property is shown in the appendix by D. S. Farley to be equivalent to the well-known
property of having a proper action on a space with walls. 
\end{abstract}

{\footnotesize
\thanks{Supported in part by NSF Grant %s DMS--0245602 and 
DMS--0504176 \newline 
{\hspace*{15pt}}2000 Math.\ Subject Class. Primary 
20F65, 22D10, 54E45 %\newline
%{\hspace*{30pt}}
%Secondary 
%\tiny File: tung.tex. Today: \today
}
}
\tableofcontents

%%%%%%%%%%%%%%%%%%%%%%%%%%%%%%%%%%%%%%%%%%%%%%%%%%%%%%%%%%%%%%%%%%%%%%%%%%%%%%%%%%%%%%%%%%%%%%%%%%%%%%%%%%%%%%%%%%%%%%%55

\section{Introduction}      

This paper is motivated by D. Farley's theorem \cite{Far} that
R. Thompson's famous infinite, finitely presented, simple group $V$ has 
the Haagerup property.
Farley's result and method are extended here to a new class of countable, discrete groups, which includes many Thompson-like groups and groups of local similarities on locally rigid, compact ultrametric spaces.

A countable discrete group $\Gamma$ has the {\it Haagerup property} if there exists an isometric action $\Gamma\curvearrowright\mathcal{H}$ on some affine Hilbert space $\mathcal{H}$ such that the action is metrically proper, which means for every bounded subset $B$ of $\mathcal{H}$, the set 
$\{ g\in\Gamma ~|~ gB\cap B\neq \emptyset\}$ is finite.
The Haagerup property is also called 
{\it Gromov's a-T-menability property}.
We refer to 
Cherix, Cowling, Jolissaint, Julg, and Valette \cite{CCJJV}
for a detailed discussion of the Haagerup property.

One reason for interest in the Haagerup property is that Higson and 
Kasparov \cite{HK} proved that the Baum--Connes conjecture 
with coefficients
is true for groups with the Haagerup property. 

The groups for which we verify the Haagerup property come with
actions on compact ultrametric spaces.
Examples of such spaces are the end spaces, or boundaries, of rooted, locally finite simplicial trees. 
See Section~\ref{section:ultrametric}, especially Remark~\ref{remark:trees}, for more details.

The actions of the groups on compact ultrametric spaces are by local similarities. There is a finiteness condition on the local restrictions
of these local similarities. See Section~\ref{section:FinitelyDetermined}
for the precise definitions.

The following is the main result of this paper.

\begin{theorem}
\label{thm:MainTheorem} If $\Gamma$ is a locally finitely determined group
of local similarities on a compact ultrametric space $X$,
then $\Gamma$ has the Haagerup property.
\end{theorem}

Examples of groups satisfying the hypothesis of Theorem~\ref{thm:MainTheorem} 
are given in Section~\ref{section:examples}.
These include Thompson's groups ($F$, $T$, and $V$) as well as other
Thompson-like groups. Moreover, if $X$ is a locally rigid, compact ultrametric space, then the full group $LS(X)$ of all local similarities
on $X$ is shown to satisfy the hypothesis. Such spaces include the end spaces of rigid trees in the sense of Bass and Lubotzky \cite{BaL}
with many interesting examples constructed by Bass and Kulkarni
\cite{BaK} and Bass and Tits \cite{BaT}.
See Hughes \cite{HugTUNG} for more on locally rigid ultrametric
spaces. 

Theorem~\ref{thm:MainTheorem} is proved in Section~\ref{sec:the main construction} by showing that the given action of $\Gamma$ on $X$ induces a zipper action of $\Gamma$ on some set. 
Zipper actions are defined in Section~\ref{section:zipper}. This concept
is implicit in Farley \cite{Far} and is a special case of Valette's characterization of the Haagerup property for countable, discrete groups \cite[Proposition 7.5.1]{CCJJV}.

In the appendix, Farley provides a proof that zipper actions are equivalent to proper actions on spaces with walls,
a well-known sufficient condition for the Haagerup property (see
Cherix et al.\ \cite[Section 1.2.7]{CCJJV}).
In addition to \cite{Far}, Farley has a separate proof \cite{Far2,Far3}, using this condition, that Thompson's groups have the Haagerup property. 
See
Cherix, Martin, and Valette \cite{CMV} for a characterization of the 
Haagerup property for countable, discrete groups in terms of spaces of measured walls.
One should also note the similarity of zipper actions with the criterion of Sageev \cite{Sag}
for a group pair to be multi--ended.
Example~\ref{example:NoWalls} shows that zipper actions do not naively 
lead to spaces with walls. 

\bigskip
\noindent
{\bf Acknowledgments.} I have benefited from conversations with
Dan Farley, Slava Grigorchuk, Mark Sapir,
Shmuel Weinberger, and Guoliang Yu.

%%%%%%%%%%%%%%%%%%%%%%%%%%%%%%%%%%%%%%%%%%%%%%%%%%%%%%%

\section{Ultrametric spaces and local similarities}
\label{section:ultrametric}

This section contains some background on ultrametric spaces and
local similarities.

\begin{definition} An {\it ultrametric space} is a metric space $(X,d)$ such that $d(x,y)\leq\max\{d(x,z), d(z,y)\}$ for all
$x,y,z\in X$.
\end{definition}

Classical examples of ultrametrics arise from $p$-adic norms, where $p$ is a prime.
For example, the $p$-adic norm $\vert\cdot\vert_p$ on the integers $\bz$ is defined
by $\vert x\vert_p := p^{-\max\{ n\in\bn\cup\{ 0\} ~\vert~ p^n ~\text{divides $x$} \}}$. The corresponding metric on $\bz$
is an ultrametric.
For more on the relationship between ultrametrics and $p$-adics, see Robert \cite{Rob}.

For the purposes of this paper, the most important examples of ultrametrics arise as end spaces of trees, which are recalled in the following example.

\begin{example} 
\label{example:ends}
Let $T$ be a {\it locally finite simplicial tree}; that is, $T$ is the geometric realization of
a locally finite, one-dimensional, simply  connected, simplicial complex.
There is a natural unique metric $d$ on $T$ such that 
$(T,d)$ is an $\br$-tree\footnote{An {\it $\br$-tree} is a metric space $(T,d)$ that is uniquely arcwise connected, and for any two 
points $x,y\in T$ the unique arc from $x$ to $y$ is
isometric to the subinterval $[0, d(x,y)]$ of $\br$.}, 
every edge is isometric
to the closed unit interval $[0,1]$, and the distance between distinct vertices $v_1, v_2$ is the minimum number
of edges in a sequence of edges $e_0, e_1,\dots, e_n$ with
$v_1\in e_0$, $v_2\in e_n$ and $e_i\cap e_{i+1}\not=\emptyset$ for $0\leq i\leq n-1$.
%It follows that $(T,d)$ is a proper $\br$-tree.
Whenever we refer to a locally finite simplicial tree $T$, the metric $d$ on $T$ will be understood
to be the natural one just described.
Choose a {\it root} (i.e., a base vertex) $v\in T$ and define
the {\it end space} of  $(T,v)$  by
$$end(T,v) = \{ x\co [0,\infty)\to T ~|~ x(0)=v \text{ and } x \text{ is an isometric 
embedding}\}.$$
For $x,y\in end(T,v)$, define
$$d_e(x,y) = \left\{ \begin{array}{ll}
0 & \text{if $x=y$}\\
1/e^{t_0} &\text{if $x\not= y$ and $t_0=\sup\{ t\geq 0 ~|~ x(t)=y(t)\}$}. \end{array}\right.
$$
It follows that $(end(T,v),d_e)$ is a 
compact ultrametric space of diameter $\leq 1$.
%The elements of $end(T,v)$ are called {\it ends} of $(T,v)$. 
\end{example}

\begin{remark}
\label{remark:trees}
There is a well-known relationship between trees and ultrametrics.
For example,
if $(X,d)$ is a compact ultrametric space, then there exists 
a rooted, locally finite simplicial tree $(T,v)$ and a 
homeomorphism $h\co [0,\infty)\to [0,\infty)$ such that $(X, hd)$ is isometric to $end(T,v)$. 
Moreover, every compact ultrametric space $(X,d)$ of diameter $\leq 1$
is isometric to the endspace of a rooted, proper $\br$-tree $(T,v)$, but not necessarily one whose edges have length less than $1$.
See Hughes \cite{Hug} and \cite{HugTUNG} for more details and further references. 
\end{remark}

If $(X,d)$ is a metric space, $x\in X$ and $\epsilon >0$, then we use the notation
$B(x,\epsilon)=\{ y\in X ~|~ d(x,y)<\epsilon\}$ for the {\it open ball about $x$ of radius
$\epsilon$}, and $\bar B(x,\epsilon)=\{ y\in X ~|~ d(x,y)\leq\epsilon\}$ 
for the {\it closed ball about $x$ of radius
$\epsilon$}.

In an ultrametric space, if two balls intersect, then one must contain the other. Moreover, closed balls are open sets and open balls are closed
sets. In the compact case, there is the following result, the proof of which is elementary and is left to the reader.

\begin{lemma}
\label{lemma:OpenClosed}
If $X$ is a compact ultrametric space and
$Y\subseteq X$, then the following are equivalent:
\begin{enumerate} 
\item $Y$ is open and closed.
\item $Y$ is a finite union of open balls in $X$.
\item $Y$ is a finite union of closed balls in $X$. $\quad \square$
\end{enumerate}
\end{lemma}

We conclude this section with the basic definitions concerning local similarities.

\begin{definition}
If $\lambda > 0$, then a map $g\co X\to Y$ between metric spaces
$(X,d_X)$ and $(Y,d_Y)$
is a {\it $\lambda$-similarity} provided $d_Y(gx, gy)=\lambda d_X(x,y)$
for all $x,y\in X$.
\end{definition}

\begin{definition} A homeomorphism $g\co X\to X$ between metric spaces is
a {\it local similarity} if for every $x\in X$ there exists $r, \lambda >0$
such that $g$ restricts to a surjective $\lambda$-similarity 
$g|\co B(x,r)\to B(gx,\lambda r)$.
\end{definition}

\begin{definition} For a metric space $X$, $LS(X)$ denotes the {\it group
of all local similarities from $X$ onto $X$}.
\end{definition}

We will be concerned with the group $LS(X)$ only when $X$ is a compact ultrametric space. It has a natural topology (the compact-open topology),
but in this paper we always endow subgroups of $LS(X)$ with the discrete
topology. 

%%%%%%%%%%%%%%%%%%%%%%%%%%%%%%%%%%%%%%%%%%%%%%%%%%%%%%%%%%%%5

\section{Locally finitely determined groups of local similarities}
\label{section:FinitelyDetermined}

In this section, we introduce the groups that are the object of study in this paper and establish some of their elementary properties.
Throughout this section,
let $X$ be a compact ultrametric space with ultrametric $d$.
The groups  are defined in terms of an extra structure on $X$, which we now define.

\begin{definition}
A {\it finite similarity structure for $X$}
is a function, denoted by $\sm$, that assigns to
each ordered pair $B_1, B_2$ of closed balls in $X$
a (possibly empty)
finite set $\sm(B_1,B_2)$ of surjective similarities
$B_1\to B_2$ such that whenever 
$B_1, B_2, B_3$ are closed balls in $X$, the following properties
hold:
\begin{enumerate}
\item (Identities) $\id_{B_1}\in\sm(B_1,B_1)$.
\item (Inverses) If $h\in\sm(B_1,B_2)$, then $h^{-1}\in\sm(B_2,B_1)$.
\item (Compositions) 
If $h_1\in\sm(B_1,B_2)$ and $h_2\in\sm(B_2,B_3)$, then
$h_2h_1\in\sm(B_1,B_3)$.
\item (Restrictions)
If $h\in\sm(B_1,B_2)$ and $B_3\subseteq B_1$, 
then $h|B_3\in\sm(B_3,h(B_3))$.
\end{enumerate}
\end{definition}

When it is necessary to indicate the dependence of $\sm$ on $X$, the notation $\sm_X$ is used. 

The word {\it finite} is used here to describe the similarity structure, not to imply that there only finitely many similarities involved 
(in general, there are infinitely many), rather to emphasize that given any two closed balls only finitely many similarities between them are chosen by
$\sm$.

\begin{example}
The {\it trivial} finite similarity structure on $X$ is given
by
$$\sm(B_1, B_2) = \begin{cases}
\{\id_{B_1}\} & \text{if $B_1=B_2$}\\
~~~\emptyset & \text{otherwise}\end{cases}$$
for each pair of closed balls $B_1, B_2$ in $X$.
\end{example}

More examples are given in the next section.

\begin{definition} Let $B$ be a closed ball in $X$.
A function $g\co B\to X$ is a {\it local similarity embedding}
if for each $x\in B$ there exist $r,\lambda > 0$ such that
$\bar B(x,r)\subseteq B$ and 
$g|\co\bar B(x,r)\to \bar B(gx,\lambda r)$ is a surjective
$\lambda$-similarity. If the choices can be made so that
$g|\in\sm(\bar B(x,r),\bar B(gx,\lambda r))$, then $g$ {\it
is locally determined by $\sm$}.
\end{definition}

\begin{definition}
\label{definition:FinitelyDetermined}
A subgroup $\Gamma$ of $LS(X)$ is {\it locally determined}
by the finite similarity structure $\sm$ for $X$
if every $g\in\Gamma$ is locally determined by $\sm$.
In this case, the group $\Gamma$ is said to be a {\it locally finitely determined
group of local similarities on $X$}.
\end{definition}

\begin{remark} Given the finite similarity structure $\sm_X$, there is a unique largest subgroup
$\Gamma\leq LS(X)$ such that $\Gamma$ is locally determined by $\sm_X$. It is defined by
$$\Gamma = \{ g\in LS(X) ~|~ g \text{ is locally determined by } \sm_X\}.$$
\end{remark}

Throughout the rest of this section,
let $\Gamma\leq LS(X)$ be a group locally determined by the
finite similarity structure $\sm$. Recall that $\Gamma$ is given the discrete topology.

\begin{definition}
A {\it region for $g\in\Gamma$} is a closed ball $B$ in $X$
such that $g(B)$ is a ball and $g|B\in\sm(B, g(B))$.
A region $B$ for $g\in\Gamma$ is a {\it maximum region for
$g$} if it is not properly contained in any region for $g$.
\end{definition}

\begin{lemma}
For each $g\in\Gamma$ and for each $x\in X$ there exists
a unique maximum region $B$ for $g$ such that $x\in B$.
\end{lemma}

\begin{proof}
By definition, $x$ is contained in some region $R$ for $g$.
Compactness of $X$ implies $R$ is contained in only finitely
many closed balls of $X$. Thus, there is a largest (with respect to set inclusion) such ball $B$ that is a 
region for $g$, and it must be a maximum region. It is the unique
maximum region for $g$ containing $x$ because any two intersecting
balls of $X$  have the property that one contains the other.
\end{proof}

It follows that for each $g\in\Gamma$, the maximum regions of
$g$ form a partition of $X$
(that is, the maximum regions of $g$ cover $X$ and are mutually disjoint), and any closed ball in $X$ contains, or is
contained in, a maximum region of $g$.

\begin{definition} If $g\in\Gamma$, then the {\it maximum
partition for $g$} is the partition of $X$ into the maximum regions
of $g$.
\end{definition}

Thus, any partition of $X$ into regions for an element $g\in\Gamma$ refines the
maximum partition for $g$.

The following lemma follows immediately from the definitions and
the Inverses Property.

\begin{lemma} If $g\in\Gamma$ and $R$ is a region for $g$, then
$g(R)$ is a region for $g^{-1}$. In addition, if $R$ is a maximum
region for $g$, then $g(R)$ is a maximum region for $g^{-1}$.
$\quad\square$
\end{lemma}

\begin{lemma}
\label{lemma:MaxPartition} Let $\mathcal P_+$ and $\mathcal P_-$
be two partitions of $X$ into closed balls.
The set 
\begin{align*}
\Gamma(\mathcal P_\pm) 
=
 \{ g\in \Gamma ~|~ 
&\mathcal{P}_+  \text{ is the maximum partition for } g \text{ and }\\
&\mathcal P_-  \text{ is
the maximum partition for } g^{-1}
\}\end{align*}
is finite.
\end{lemma}

\begin{proof}
Say $\mathcal{P}_+ =\{ B_1,\dots, B_n\}$ where $n=|\mathcal{P}_+|$ is the cardinality of
$\mathcal{P}_+$.
Let $\bi(\mathcal{P}_+,\mathcal{P}_-)$ denote the finite set of bijections from
$\mathcal{P}_+$ to $\mathcal{P}_-$.
For $h\in \bi(\mathcal{P}_+,\mathcal{P}_-)$, let $S_h :=
\prod_{i=1}^n \sm(B_i,h(B_i))$ and note that $S_h$ is finite. 
Define the finite set 
$F$ to be the disjoint union $F:=\coprod_{h\in\bi(\mathcal{P}_+,\mathcal{P}_-)}S_h$,
which we prefer to write as
$F= \bigcup_{h\in\bi(\mathcal{P}_+,\mathcal{P}_-)}(h,S_h)$.
If $g\in\Gamma(\mathcal P_\pm)$, then $g_\ast\in\bi(\mathcal{P}_+,\mathcal{P}_-)$ is defined by
$g_\ast(B)=g(B)$ for all $B\in\mathcal P_+$.
Clearly, there is an injection $\Gamma(\mathcal P_\pm)\to F$
given by $g\mapsto (g_\ast, (g|B_1,\dots,g|B_n))$.
\end{proof}

Recall that if $\mathcal{P}$ and $\mathcal{Q}$ are two collections, then
$\mathcal{P}$ {\it refines} $\mathcal{Q}$ means for every $P\in\mathcal{P}$ there exists $Q\in\mathcal{Q}$ such that $P\subseteq Q$.

\begin{lemma}
\label{lemma:RefineMax}
Let $\mathcal P_+$ and $\mathcal P_-$
be two partitions of $X$ into closed balls.
The set 
\begin{align*}
\Gamma_{\mathrm{ref}}(\mathcal P_\pm)
=
\{ g\in \Gamma ~|~ &\mathcal P_+ \text{ refines
the maximum partition for } g \text{ and }\\
&\mathcal P_- \text{ refines
the maximum partition for } g^{-1}
\}\end{align*} 
is finite.
\end{lemma}

\begin{proof}
Given any closed ball $B$ in  $X$, there exist only finitely many
distinct closed balls of $X$ containing $B$. Hence, any partition of $X$ into closed balls refines only finitely many other partitions of $X$ into
closed balls. Thus, there exist only finitely many pairs, say
$(\mathcal{P}_+^i,\mathcal{P}_-^i)$ for
$i=1,\dots,n$, of partitions of $X$ into closed balls such that 
$\mathcal{P}_+$ refines $\mathcal{P}_+^i$ and
$\mathcal{P}_-$ refines $\mathcal{P}_-^i$ for all $i=1,\dots,n$.
Clearly, 
$\Gamma_{\mathrm{ref}}(\mathcal P_\pm)=\bigcup_{i=1}^n\Gamma(\mathcal{P}_\pm^i)$
and the result follows from Lemma~\ref{lemma:MaxPartition}.
\end{proof}

\begin{lemma}
$\Gamma$ is countable.
\end{lemma}

\begin{proof}
$X$ has only countably many closed balls; hence, $X$ has only countably many
partitions into closed balls and only countably many pairs,
say $(\mathcal{P}_+^i, \mathcal{P}_-^i)$ for $i=1,2,3,\dots$, of partitions of $X$ into
closed balls. 
Clearly,
$\Gamma=\bigcup_{i=1}^\infty\Gamma(\mathcal{P}_\pm^i)$
and the result follows from Lemma~\ref{lemma:MaxPartition}.
\end{proof}

%%%%%%%%%%%%%%%%%%%%%%%%%%%%%%%%%%%%%%%%%%%%%%%%%%%%%
\section{Examples}
\label{section:examples}

In this section we give examples of locally finitely determined groups of local
similarities on compact ultrametric spaces.
The examples include Thompson's groups so that Farley's result \cite{Far}
is recovered from Theorem~\ref{thm:MainTheorem}.
The examples also include many other Thompson-like groups, as well as the full local similarity groups of end spaces of certain trees constructed by Bass and Kulkarni \cite{BaK}  and Bass and Tits \cite{BaT}.

We begin by recalling standard alphabet language and notation.
An excellent reference is Nekrashevych \cite{Nek}.
An {\it alphabet} is a non-empty finite set $A$ and finite (perhaps empty), ordered subsets
of $A$ are {\it words}. 
The set of all words is denoted $A^*$
and the set of {\it infinite words} is denoted
$A^\omega $; that is,
$$A^* = \coprod_{n=0}^\infty A^n ~~~\text{and}~~~ 
A^\omega = \prod_1^\infty A.$$
Let $T_A$ be the tree associated to $A$. The vertices of
$T_A$ are words in $A$; two words $v,w$ are connected by an edge
if and only if there exists $x\in A$ such that
$v=wx$ or $vx=w$. The root of $T_A$ is $\emptyset$.
Thus, $A^\omega = end(T_A,\emptyset)$ and so comes with a natural
ultrametric as described in Example~\ref{example:ends} making
$A^\omega$  compact.
We may assume that $A$ is totally ordered. There is then an induced
total order on $A^\omega$, namely  the lexigraphic order.

\begin{example}{\bf(The Higman--Thompson groups $G_{d,1}$)}\,
\label{example:lop}
Let $\Gamma= LS_{l.o.p}(A^\omega)$ be the subgroup of $LS(A^\omega)$ consisting of locally order
preserving local similarities on $A^\omega$, where
a map $h\co A^\omega\to A^\omega$ is {\it locally order preserving}
if for each $x\in A^\omega$ there exists $\epsilon >0$ such that
$h|\co B(x,\epsilon)\to A^\omega$ is order preserving.
We denote $\id_{A^\omega}= e$; it is the unique order preserving
isometry $A^\omega\to A^\omega$. Any closed ball in $A^\omega$ has a
unique order preserving similarity onto $A^\omega$; hence, 
if $B_1$ and $B_2$ are two closed balls in $A^\omega$, then there is
a unique order preserving similarity of $B_1$ onto $B_2$.
Let $\sm(B_1, B_2)$ consist solely of that unique order preserving similarity.
This can be described using alphabet language quite easily as follows.
A closed ball in $A^\omega$ is given by $vA^\omega$, where
$v\in A^*$. For 
$v,w\in A^*$,
$\sm(vA^\omega, wA^\omega)$ consists of the unique order
preserving similarity $vA^\omega\to wA^\omega$; $vx\mapsto wx$.
Clearly, this defines a finite similarity structure $\sm_{A^\omega}$ and
$\Gamma$ is locally determined by $\sm_{A^\omega}$. 

When the alphabet is $A=\{ 0,1\}$, we get 
Thompson's group $V = LS_{l.o.p}(A^\omega)$.
The subgroups
$F\leq T\leq V$ are also locally determined by the same finite similarity structure $\sm_{A^\omega}$ (elements of $T$ are further required to be cyclicly order preserving; those of $F$, to be order preserving).
In general, $LS_{l.o.p}(A^\omega)$ is the Higman--Thompson group
$G_{d,1}$, where $d=|A|$. For background on these groups, see
Cannon, Floyd, and Parry \cite{Can} and for other references,
see Hughes \cite[Section 12.3]{HugTUNG}.
\end{example}

\begin{example}{\bf (Generalized Higman--Thompson groups $LS_{l.o.p}(X)$)}\,
The previous example  can easily be extended to end spaces of rooted, ordered, proper  $\br$-trees $(T,v)$ so that the groups
$LS_{l.o.p}(X)$, where $X=end(T,v)$, defined in Hughes \cite[Section 12.3]{HugTUNG}, become locally finitely determined
groups of local similarities on $X$. 
In particular, it is easy to see that the Higman-Thompson groups $G_{d,n}$, $n\geq 1$, fit into this framework.
\end{example}

\begin{example}{\bf (Subgroups)}\,
A subgroup $H$ of a group $\Gamma$ of local similarities locally determined by the finite similarity structure $\sm$ is also locally determined by $\sm$. This is clear because Definition~\ref{definition:FinitelyDetermined} is a condition on elements of $\Gamma$,
which therefore holds for each element of $H$.
\end{example}

\begin{example}{\bf (Nekrashevych-R\"over groups $V_d(H)$, $H$ finite)}\,
Suppose $H$ is a finite, self-similar group over the alphabet $A$,
with $d=|A|$ (see Nekrashevych \cite{Nek}).
Nekrashevych \cite{NekJOT} defines a group $V_d(H)\leq LS(A^\omega)$
generalizing a construction of R\"over \cite{Rov99}.
To describe these groups note that there is a natural similarity from $A^\omega$ onto any closed ball of $A^\omega$; thus, any surjective similarity $h\co B_1\to B_2$ between closed balls gives rise to an isometry
$h_\ast$ of $A^\omega$:
$$h_\ast\co A^\omega \to B_1 \stackrel{h}{\rightarrow} B_2 \to A^\omega.$$
Then an element $g\in LS(A^\omega)$ is in $V_d(H)$ if and only if
for each $x\in A^\omega$ there exists $\epsilon, \lambda > 0$ such that
$g|:B(x,\epsilon)\to B(gx,\lambda\epsilon)$ is a $\lambda$-similarity
and $(g|)_\ast\in H$. 
For his general construction, Nekrashevych does not require $H$ to be finite,
but we require it in order to define the following finite similarity
structure on $A^\omega$:
if $B_1, B_2$ are closed balls of $A^\omega$, then $\sm(B_1, B_2)$ consists of all surjective similarities $h\co B_1\to B_2$ such that 
$h_\ast\in H$.
The Restrictions Property follows from 
the self-similarity property of $H$.
Clearly, $V_d(H)$ is finitely determined by $\sm_{A^\omega}$.

For example,
note that in the special case $H=\{ 1\}$, $V_d(H) = G_{d,1}$.

For a nontrivial example, let $\Sigma_d$ be the symmetric group on $A$. The action of $\Sigma_d$ on $A^\ast$ given by 
$\sigma(a_1\dots a_n)=\sigma(a_1)\dots\sigma(a_n)$ induces an action of $\Sigma_d$ on the tree $T_A$
and we let $H\cong\Sigma_d$ be the corresponding self-similar subgroup of $\Aut(T_A)$.
Note that $G_{d,1}\leq V_d(\Sigma_d)$ and that $\Gamma:= V_d(\Sigma_d)\cap\Aut(T_A)$ is a contracting self-similar
group with nucleus $\Sigma_d$ (see Nekrashevych \cite[Section 2.11]{Nek} for the definitions).

Generalizing this last observation, let $\Gamma$ be any contracting self-similar subgroup of $\Aut(T_A)$ whose nucleus $\mathcal{N}$
is a finite group (in general, contracting self-similar groups have nuclei that are finite sets---the condition that the nucleus be
a group is quite restrictive). It follows that $\mathcal{N}$ is a finite self-similar group and we can form the locally finitely determined
group $V_d(\mathcal{N})$. For each pair $B_1, B_2$ of closed balls in $A^\omega$, $\sm(B_1,B_2)$ is naturally identified with $\mathcal{N}$.
Note that $\Gamma\leq V_d(\mathcal{N})\cap\Aut(T_A)$.
\end{example}

\begin{example}{\bf (Groups acting on trees with finite vertex stabilizers)}\, 
Let $(T,v)$ be a geodesically complete, rooted, locally finite simplicial tree, where {\it geodesically complete} means no vertex, except possibly the root, has valency $1$. Let $\Gamma$ be a subgroup
of the isometry group $Isom(T)$ such that $\Gamma$ has finite vertex stabilizers (that is, for each vertex $w\in T$, the isotropy group 
$\Gamma_w$ is finite). There is a well-known homomorphism
$\epsilon\co Isom(T)\to LS(X)$, where $X=end(T,v)$, 
explicitly described in Hughes \cite[Section 12.1]{HugTUNG}.
We will show that $\epsilon(\Gamma)$ is locally finitely determined. 
If $B$ is a closed ball in $X$, then there exists a vertex 
$w_B\in T$ such that 
$B=\{ x\in X ~|~ x(d(v,w_B))=w_B\}$ and
$T_B =\{ x(t) ~|~ x\in B \text{ and } t\geq d(v,w_B)\}$ is a subtree
of $T$ with $B$ similar to $end(T_B,w_B)$.
Define a finite similarity structure $\sm$ as follows.
If $B_1, B_2$ are closed balls in $X$, let
$$\sm(B_1, B_2) =\{ \epsilon(g)|\co B_1\to B_2 ~|~
g\in \Gamma,~ g(w_{B_1})=w_{B_2}, \text{ and }
g(T_{B_1}) =T_{B_2}\}.$$
The finite vertex stabilizers assumption implies that $\sm(B_1,B_2)$ is finite. The other properties of a similarity structure are easy to verify.
Moreover, $\epsilon(\Gamma)$ is locally determined 
by $\sm$.
Note that $\epsilon$ is an injection except when $T$ is isometric to $\br$.
In particular, finitely generated free groups are locally finitely determined.
Of course, it is well-known that discrete groups acting on trees with
finite vertex stabilizers have the Haagerup property (see Cherix et al.\ \cite[Section 1.2.3]{CCJJV}).
\end{example}

\begin{example}{\bf (Local similarity groups of locally rigid, compact ultrametric spaces)}\,
Let $X$ be a locally rigid, compact ultrametric space as defined in
Hughes \cite{HugTUNG}. 
It is proved there that a compact ultrametric space $X$ is locally rigid if and only if the isometry group $Isom(X)$ is finite.
In particular, the isometry group of any closed ball in $X$ is also finite. 
From this it follows easily that for any two closed balls $B_1, B_2$ in $X$, the set of all surjective similarities
from $B_1$ to $B_2$ is finite. We can therefore define a finite similarity structure $\sm$ by letting 
$\sm(B_1,B_2)$ be the set of all  similarities from $B_1$ onto $B_2$. Then the group $\Gamma= LS(X)$ of all local similarities of $X$
onto itself is locally determined by $\sm$. 
\end{example}

\begin{example}{\bf (Local similarity groups of end spaces of rigid trees)}\, 
Let $T$ be a locally finite simplicial tree that is {\it rigid}; that is, the group of automorphisms 
$\Aut(T)$ is discrete; see Bass and Lubotzky \cite{BaL}. 
Let $X=end(T,v)$, where $v$ is a chosen vertex of $T$. Assuming that $(T,v)$ is geodesically complete,
the rigidity of $T$ is equivalent to local rigidity of
$X$; see Hughes \cite[Section 12.2]{HugTUNG}. Hence, $\Gamma := LS(X)$ is locally finitely determined as described in the preceding example.
An interesting source of examples of rigid trees come from {\it $\pi$-rigid} graphs of Bass and Kulkarni \cite{BaK} and
Bass and Tits \cite{BaT}.
These are finite, connected, simplicial graphs $G$ with the property that if $\tilde{G}$ is the universal covering tree of $G$, then
$\Aut(\tilde{G}) =\pi_1(G)$. In particular, $\tilde{G}$ is rigid
and $LS(end(\tilde{G},v))$ is finitely determined.
\end{example}

%%%%%%%%%%%%%%%%%%%%%%%%%%%%%%%%%%%%%%%%%%%%%%%%%%%%%%%%%%%%%%%%%%%%

\section{Zipper actions}
\label{section:zipper}

In this section we discuss a sufficient condition, called a zipper action, for a discrete group to have the Haagerup property.
This condition is implicit in Farley \cite{Far} and is a special
case of the necessary and sufficient condition due to
Valette \cite[Proposition 7.5.1]{CCJJV}. 
Moreover, in the appendix to this paper, Farley shows that zipper actions are equivalent to proper actions on spaces of walls. 
Apart from the terminology, there is nothing original in this section.

\begin{definition}
\label{def:ZipperAction}
A discrete group $\Gamma$ has a {\it zipper action}
if 
there is a left action 
$\Gamma\curvearrowright\mathcal E$
of  $\Gamma$ on a set $\mathcal E$ and
a subset $Z\subseteq\mathcal E$ such that
\begin{enumerate}
\item for every $g\in\Gamma$, the symmetric difference 
$gZ \,\triangle\, Z$ is finite, and
\item for every $r>0$, $\{ g\in\Gamma ~|~ |gZ\,\triangle\, Z|\leq
r\}$ is finite.
\end{enumerate}
\end{definition}

Note that if $\Gamma$ is an infinite group then
condition 2 implies $Z$ must also be infinite.
Also, the action $\Gamma\curvearrowright\mathcal E$ is not assumed to be effective; however, condition 2 implies that the kernel of the action is 
finite.

The terminology arises as follows. We think of $Z$ as being an infinite
zipper in $\mathcal E$ that is unzipped by the action of $\Gamma$.
Only a finite portion is unzipped by any finite subset of $\Gamma$, but as one takes larger finite subsets of $\Gamma$, more of $Z$ is unzipped.

\begin{example}
We show that the group $\bz$ has a zipper action.
Let 
$${\mathcal E} = \bz \hbox{ ~and~ } Z=\{ n\in{\mathcal E} ~|~ n\leq 0\}.$$
An action $\bz\curvearrowright\mathcal E$ is defined by 
$g\cdot n = g+n$.
If $g\in\bz$ and $g\geq 0$, then $Z\subseteq gZ$ and $gZ \,\triangle\, Z = \{ n\in\mathcal E ~|~ 0 < n\leq g\}$.
If $g\in\bz$ and $g\leq 0$, then $gZ\subseteq Z$ and 
$gZ \,\triangle\, Z = \{ n\in\mathcal E ~|~ g < n\leq 0\}$.
One may say ``$Z$ is taken further off itself as $g\to+\infty$ in $\bz$'' and ``$Z$ is taken deeper into itself as $g\to-\infty$ in $\bz$.''
Thus, $\vert gZ \,\triangle\, Z\vert =\vert g\vert$ for all $g\in \bz$.
If $r\geq 0$, then 
$\{ g\in\bz ~|~ \vert gZ \,\triangle\, Z\vert \leq r\}= \{ g\in\bz ~|~ \vert g\vert\leq r\}$, which is finite.
\end{example}

The  proof of the following theorem, which is a special case of 
Valette \cite[Proposition 7.5.1]{CCJJV}, is implicit in Farley \cite{Far}, but is included for
completeness.

\begin{theorem}
\label{thm:HPcriterion}
If the discrete group $\Gamma$ has a zipper action,
then $\Gamma$ has the Haagerup property.
\end{theorem}

\begin{proof}
Define $\pi:\Gamma\to\ell^\infty(\mathcal E)$ by
$\pi(g)= \chi_{gZ}-\chi_Z$ (where $\chi_Y$ denotes the characteristic
function of $Y\subseteq\mathcal E$).
Note:
\begin{enumerate}
\item The support of $\pi(g)$ is $gZ\,\triangle\,Z$; hence,
$\pi(g)$ is finitely supported and $\pi(g)$ is in the Hilbert space
$\ell^2(\mathcal E)$
for all $g\in\Gamma$.
\item The square of the $\ell^2$-norm $\Vert\pi(g)\Vert_2^2 
= |gZ\,\triangle\,Z|$
for all $g\in\Gamma$.
\item
\label{finite}
For every $r>0$, $\{g\in\Gamma ~|~ \Vert\pi(g)\Vert_2\leq r
\}$ is finite.
\end{enumerate}

The action of $\Gamma$ on $\mathcal E$ induces a unitary
left action of $\Gamma$ on $\ell^2(\mathcal E)$,
$\rho:\Gamma\to\mathcal B(\ell^2(\mathcal E))$, where
$\mathcal B(\ell^2(\mathcal E))$ is the space of bounded linear operators
on $\ell^2(\mathcal E)$. Namely,
$\rho(g)(f)(e) = f(g^{-1}e)$ for $g\in\Gamma$, $f\co\mathcal{E}\to\bc$ in
$\ell^2(\mathcal E)$, and $e\in\mathcal{E}$.

One checks that $\pi$ is a $1$-cocycle for $\rho$; that is,
$\pi(g_1g_2)=\rho(g_1)\pi(g_2)+\pi(g_1)$ for all $g_1, g_2\in\Gamma$.
For this, it is useful to observe that
$g\chi_Y=\chi_{gY}$ in $\ell^\infty(\mathcal E)$ for any
$Y\subseteq\mathcal E$.
It follows that $A:\Gamma\to Isom(\ell^2(\mathcal E))$ defined
by $A(g)(f) = \rho(g)(f) +\pi(g)$ is an affine isometric action
of $\Gamma$ on $\ell^2(\mathcal E)$.
Moreover, property \ref{finite} above guarantees that $A$ is
metrically proper.
\end{proof}

%\begin{remark} It seems unlikely that the converse of 
%Theorem~\ref{thm:HPcriterion} holds, but I do not know an example.
%\end{remark}

\begin{remark} The existence of a zipper action is preserved by direct sums of groups. 
For let $\Gamma_i$ ($i=1,2$)  be discrete groups having  left actions
$\Gamma_i\curvearrowright\mathcal E_i$ and subsets $Z_i\subset \mathcal E_i$ as in Definition~\ref{def:ZipperAction}.
Let $\Gamma:=\Gamma_1\oplus\Gamma_2$, $\mathcal E:=\mathcal E_1\amalg \mathcal E_2$, $Z:= Z_1\amalg Z_2$, and define a left action
$\Gamma\curvearrowright\mathcal E$ in the obvious way: 
$(g_1,g_2)\cdot e_i = g_ie_i$ where $e_i\in\mathcal E_i$ and $i\in\{ 1,2\}$.
The conditions are readily checked.
\end{remark}

%%%%%%%%%%%%%%%%%%%%%%%%%%%%%%%%%%%%
\section{The main construction}
\label{sec:the main construction}

In this section we prove the following theorem.

\begin{theorem}
\label{thm:ZipperExistence} 
If $\Gamma$ is a locally finitely determined group
of local similarities on a compact ultrametric space $X$, then $\Gamma$ has a zipper
action.
\end{theorem}

Clearly, Theorem~\ref{thm:MainTheorem} follows from 
Theorems~\ref{thm:HPcriterion} and \ref{thm:ZipperExistence}.

Throughout this section, $X$ will 
denote a compact 
ultrametric space and
$\Gamma\leq LS(X)$ will be a group locally determined by a
finite similarity structure $\sm$ on $X$.

Before defining a set $\mathcal{E}$ with a zipper action
$\Gamma\curvearrowright\mathcal E$,
note that it follows from Lemma~\ref{lemma:OpenClosed} that the image of a local similarity embedding $f\co B\to X$, where $B$ is a closed ball in $X$, is a finite union of mutually disjoint
closed balls in $X$.

Now let $\mathcal E$ be the set of equivalence classes of pairs $(f,B)$
where $B$ is a closed ball in $X$ and $f\co B\to X$ is a local similarity
embedding locally determined by $\sm$.
Two such $(f_1, B_1)$ and $(f_2, B_2)$ are {\it equivalent}
provided there exists $h\in\sm(B_1, B_2)$ such
that $f_2h=f_1$ (in particular,
$f_1(B_1)=f_2(B_2)$). The verification that this is an equivalence
relation requires the Identities, Compositions, and Inverses
Properties of the similarity structure.
Equivalence classes are denoted by $[f,B]$.

Let $Z = \{[f,B]\in\mathcal E ~|~ f(B)$ ~\text{is a closed ball in $X$
and $f\in\sm(B, f(B))$}\}.

Note that an element $[f,B]\in Z$ is uniquely determined
by the closed ball $f(B)$. In fact, $[f,B] = [\incl_{f(B)},f(B)]$,
where $\incl_Y\co Y\to X$ denotes the inclusion map.
Thus,
$$Z = \{[\incl_B, B]\in{\mathcal E} ~|~ B ~\text{is a closed ball in $X$}\}.$$
In particular, $Z$ can be identified with the collection of all closed balls in $X$.

There is a left action 
$\Gamma\curvearrowright\mathcal E$ defined by
$g[f,B] = [gf,B]$. The fact that $[gf,B]\in\mathcal{E}$ follows from the
Compositions and Restrictions Properties of the similarity structure. 

It follows from the description of $Z$ above that for all $g\in \Gamma$,
$$gZ = \{ [g|_B,B]\in{\mathcal E} ~|~ B ~\text{is a closed ball of $X$}\}.$$

The next part of this section is devoted to establishing,
in Corollary~\ref{cor:ZipperOne}
and Lemma~\ref{lemma:ZipperTwo} below, the two properties 
required of a zipper action.

\begin{lemma} Let $B$ be a closed ball in $X$ and $g\in\Gamma$.
Then $[\incl_B,B]\in Z\setminus gZ$ if and only if $B$ properly
contains a maximum region of $g^{-1}$.
\end{lemma}

\begin{proof}
Suppose first that $[\incl_B,B]\in Z\setminus gZ$ and, by way of
contradiction, there exists a maximum region $R$ for $g^{-1}$
containing $B$. Then $g^{-1}R$ is a ball and $g^{-1}|R\in
\sm(R,g^{-1}R)$. The Restrictions Property implies 
$g^{-1}|B\in\sm(B,g^{-1}B)$ and
$[g^{-1}|B,B]\in Z$. Clearly,
$[\incl_B,B]=g[g^{-1}|B,B]\in gZ$, which is a contradiction.

Conversely, let $R$ be a maximum region of $g^{-1}$ properly
contained in $B$. If $[\incl_B,B]\in gZ$, then 
there exists $[\incl_{B_1}, B_1]\in Z$ such that 
$[g|B_1,B_1]=g[\incl_{B_1},B_1]=[\incl_B,B]$, which is to say
$g(B_1)=B$. 
Moreover,
$[g|B_1, B_1] = [\incl_B,B]$ implies that $g|\co B_1\to B$ is in 
$\sm(B_1,B)$.
The Inverses Property implies $g^{-1}|\co B\to B_1$ is in $\sm(B,B_1)$.
In particular, $B$ is  a region for $g^{-1}$, contradicting the maximality of $R$. Thus, $[\incl_B,B]\notin gZ$.
\end{proof}

\begin{lemma}
\label{lemma:bijection} For each $g\in \Gamma$, 
the function $[\incl_B,B]\mapsto B$ is a bijection
from $Z\setminus gZ$ to the set of closed balls of $X$ properly
containing maximum regions of $g^{-1}$.
Moreover, the function
$g[\incl_B,B]\mapsto B$ is a bijection
from $gZ\setminus Z$ to the set of closed balls of $X$ properly
containing maximum regions of $g$.
\end{lemma}

\begin{proof}
The first statement follows immediately from the preceding
lemma. The second follows from the first together with the
observation that $g[\incl_B,B]\mapsto [\incl_B,B]$ is a bijection
from $gZ\setminus Z$ to $Z\setminus g^{-1}Z$.
\end{proof}

\begin{corollary}
\label{cor:ZipperOne}
For each $g\in\Gamma$, the symmetric difference
$gZ\,\triangle\, Z$ is finite.
\end{corollary}

\begin{proof}
This follows immediately from the preceding lemma because there
are only a finite number of closed balls of $X$ containing a maximum
region of $g$ or $g^{-1}$.
\end{proof}

\begin{lemma}
\label{lemma:ZipperTwo}
For each $r>0$, $\{ g\in\Gamma ~|~ |gZ\,\triangle\, Z|\leq
r\}$ is finite.
\end{lemma}

\begin{proof}
Let $\Gamma_r = \{ g\in\Gamma ~|~ |gZ\,\triangle\,Z|\leq r\}$.
Since $|gZ\,\triangle\,Z|=|g^{-1}Z\,\triangle\,Z|$,
$g\in\Gamma_r$ if and only if $g^{-1}\in\Gamma_r$.
For each $x\in X$, let 
$$M_{r,x} =
\{ R ~|~ R \hbox{ is a maximum region for some } g\in\Gamma_r \hbox{ and } x\in R\}.$$
By Lemma~\ref{lemma:bijection} if $g\in\Gamma_r$, then 
the number of closed balls of $X$ properly containing a maximum region of $g$
is less than or equal to $r$.
In particular, if $R\in M_{r,x}$, there are at most $r$ closed balls of
$X$ properly containing $R$. Since $M_{r,x}$ is totally ordered by inclusion,
it follows that $M_{r,x}$ is finite and 
there exists $R_{r,x}\in M_{r,x}$ such that $R_{r,x}\subseteq R$
for all $R\in M_{r,x}$.
The set 
$\mathcal{P}_r := \{ R_{r,x} ~|~ x\in X\}$ is a partition of $X$ and each 
$R_{r,x}$ is a region for $g$ and for $g^{-1}$, for all $g\in \Gamma_r$.
That is to say $\mathcal{P}_r$ refines the maximum partitions of 
both $g$ and $g^{-1}$ for all $g\in \Gamma_r$.
Setting $\mathcal{P}_+ =\mathcal{P}_r=\mathcal{P}_-$,
this means $\Gamma_r\subset\Gamma_{\mathrm{ref}}(\mathcal P_\pm)
$
and the result follows from Lemma~\ref{lemma:RefineMax}.
\end{proof}

The proof of Theorem~\ref{thm:ZipperExistence} is now complete.

Even though it will be shown in the appendix that 
a zipper action gives rise to a proper action on a space with walls, the next example is included to show
that this does not happen  in the most
naive way.
See the appendix for definitions and references.

\begin{example}
\label{example:NoWalls}
Consider the alphabet $A=\{ 0,1\}$ and Thompson's group $V\leq LS(A^\omega)$ as in Example~\ref{example:lop}.
The construction above gives sets $Z\subseteq\mathcal{E}$ and a
zipper action $V\curvearrowright\mathcal{E}$. 
It might be expected that $\mathcal{W} :=\{ (gZ,\mathcal{E}\setminus gZ) ~|~ g\in V\}$ is a set of walls for $\mathcal{E}$.
However, this is not the case. Specifically, we show there exists
$[f_1,B_1], [f_2,B_2]\in\mathcal{E}$ and an infinite subset
$G\subseteq V$ such that 
$[f_1,B_1]\in gZ$ and $[f_2,B_2]\notin gZ$ for every $g\in G$ 
(that is, there are two elements of $\mathcal{E}$ separated by infinitely many
walls).
Let $B_1=0A^\omega$ and $B_2=1A^\omega$.
Let $f_1\co B_1\to A^\omega$ be the inclusion and let $f_2\co B_2\to A^\omega$ be defined by
$$\begin{cases}
f_2(10w) = & 10w \\
f_2(11w) = & 111w \end{cases},
\text{~~~for all $w\in A^\omega$.}$$
Let $G=\{ g\in V ~|~ g \text{~~is a local isometry and $g|B_1=f_1\}$.}$
The required conditions are readily checked. (Note also that
$[f_2,B_2]\in VZ$ so that a set of walls will not result by reducing the size of $\mathcal{E}$.)
\end{example}

%

%%%%%%%%%%%%%%%%%%%%%%%%%%%%%%%%%%%%%%%%%%%%%%%%%%%%%%%%%%%%%%%%%%%%%%%%%%%%%%%%%%%%%%%%
%%%%%%%%%%%%Appendix by Farley
%%%%%December 2, 20007

%%%%%%%%%%%%Appendix by Farley

\appendix
\section{Appendix by Daniel S. Farley: Zipper actions, spaces with walls,
and CAT($0$) cubical complexes}

The purpose of this appendix is to show that the property of having a
zipper action is equivalent to having a proper action on a space with walls.
Spaces with walls were introduced by Haglund and Paulin \cite{HagPau}, who
wanted a common language for describing a range of combinatorial structures,
among them CAT($0$) cubical complexes.  Sageev \cite{Sag} had in effect
shown that
CAT($0$) cubical complexes are spaces with walls, so a group acting
properly on a CAT($0$) cubical complex also acts properly on a space with
walls   
(see Cherix et al.\ \cite[Section 1.2.7]{CCJJV} for a
discussion about the relevance of this to the Haagerup property and for
more references).  Chatterji and Niblo \cite{ChN} and Nica \cite{Nica}
proved the converse, namely, that a group that acts properly on a space
with walls also acts properly on a CAT($0$) cube complex. Thus, having a
proper action of a group $\Gamma$ on a CAT($0$) cube complex is
equivalent to having a zipper action of $\Gamma$. In fact, a zipper is
closely related to Sageev's notion of an almost invariant set in
\cite{Sag} and the  the discussion in this appendix is implicit in
\cite{Sag}. At any rate, the experts will find this result familiar.

\begin{theorem}
\label{thm:appendix}
A discrete group $\Gamma$ has a zipper action if and only if $\Gamma$
acts properly on a space with walls.
\end{theorem}

\begin{definition}
Let $S$ be a set and let $\Gamma$ be a group.
\begin{enumerate}
\item A {\it wall in $S$} is a pair $W=\{ H_1, H_2\}$ such that
$W$ is a partition of $S$; i.e., $S= H_1\cup H_2$, $H_1\cap
H_2=\emptyset$, and $H_1\not=\emptyset\not= H_2$.
\item If $W = \{ H_1, H_2\}$ is a wall in $S$, then $H_1$ and
$H_2$ are called {\it half-spaces} in $S$.
\item Elements $x, y\in S$ are {\it separated by} the wall $W$ if $x\in
H_1$ and $y\in H_2$.
\item The set $S$ is a {\it space with walls} if there is given a set of
walls in $S$ such that for every
$x, y\in S$ there are at most finitely many walls separating $x$ and $y$.
In this case, define
$d(x,y)$ to be the number of walls separating $x$ and $y$ and note that
$d$ is a pseudometric (that is, it is symmetric and satisfies the
triangle inequality).
\item The group $\Gamma$ {\it acts on the space $S$ with walls} if there
is an action of $\Gamma$ on $S$ such that $\Gamma$ permutes the walls. 
In this case the action of $\Gamma$ is by isometries of the pseudometric.

\item The group $\Gamma$ is said to act {\it properly on the space $S$
with walls} if for all $r\in\br$ and for all $p\in S$,
$\{ g\in\Gamma ~|~ d(gp, p)< r\}$ is finite.
\end{enumerate}
\end{definition}

Note that a finite group $\Gamma$ has a zipper action (let $\mathcal E =
Z = \Gamma$) and also acts properly on a space with walls
(let $S$ be the set of all partitions of $\Gamma$ into two nonempty,
disjoint subsets). Therefore, in the proof of
Theorem~\ref{thm:appendix}, it is assumed that $\Gamma$ is infinite.

\bigskip
\begin{proof*} {\it of Theorem~\ref{thm:appendix}.}
Suppose first that there is a zipper action
$\Gamma\curvearrowright\mathcal E$ with zipper $Z\subseteq\mathcal E$.
Define
$$A = \bigcap_{g\in\Gamma}gZ ~\text{ and }~ 
B=\bigcap_{g\in\Gamma}\left(gZ\right)^{c}.$$
Note that $\mathcal E\setminus(A\cup B)\not=\emptyset$, for otherwise
$Z \triangle gZ = \emptyset$ for all $g \in \Gamma$, and so $Z$ would
fail to be a zipper.
Let $S=\{ gZ\subseteq \mathcal E ~|~ g\in\Gamma\}$.
If $x\in \mathcal E$, define $H_x^+ =\{ gZ\in S ~|~ x\in gZ\}$ and
$H_x^- =\{ gZ\in S ~|~ x\notin gZ\}$. Then $\{ \{ H_x^+, H_x^- \}  ~|~
x\in\mathcal E\setminus(A\cup B)\}$ is a set of walls for $S$.
The condition that $x\notin A\cup B$ ensures that
$H_x^+\not=\emptyset\not= H_x^-$.
To see that for $g_1, g_2\in \Gamma$,
there are only finitely many walls separating $g_1Z$ and $g_2Z$,
observe that $\{ H_x^+, H_x^- \} $ separates
$g_1Z$ and $g_2Z$ if and only if $x\in g_1Z\, \triangle \, g_2Z$, and
$x\in g_1Z\, \triangle \, g_2Z\subseteq
(g_1Z\, \triangle \, Z)\cup  (g_2Z \, \triangle \, Z)$, which is finite.
The action $\Gamma\curvearrowright S$ is given by $h\cdot gZ = hgZ$.
Since $hH_x^\pm = H_{hx}^\pm$ for all $h\in\Gamma$ and $x\in\mathcal
E\setminus (A\cup B)$ (and $x\in A\cup B$ if and only if
$hx\in A\cup B$), it follows that $\Gamma$ permutes the walls.
Since, by the second property of a zipper action (Definition~\ref{def:ZipperAction}), there
are, for a given $r > 0$, at most finitely many $g \in \Gamma$
such that $d(gZ, Z) = |gZ \triangle Z| < r$, 
the action
of $\Gamma$ is proper.  

%The observation above also implies that $d(hgz, gz) = \vert hgZ\,
%\triangle \, gZ\vert \leq
%\vert hgZ\, \triangle \, Z\vert +\vert gZ\, \triangle \, Z\vert$, which
%shows that the action is proper.

Conversely,
suppose there is a proper action $\Gamma\curvearrowright S$ of $\Gamma$
on a space $S$ with walls.
Let $\mathcal E$ be the set of all half-spaces of $S$.
For each $x\in S$, define $Z_x = \{ H\in\mathcal E ~|~ x\in H\}$.
Note that $g\cdot Z_x = Z_{gx}$ for all $g\in\Gamma$ and $x\in S$.
Fix a base point $p\in S$ and let $Z=Z_p$ be the zipper for the action.
For each $g\in\Gamma$, it follows that
\begin{eqnarray*}
Z_{gp}\, \triangle \, Z_p  & =  &   
\{ H \in \mathcal{E} \mid gp\in H \text{ and } p\notin H\}
\cup \{ H \in \mathcal{E} \mid gp\notin H \text{ and } p\in H\} \\
& = & \{ H \in \mathcal{E} \mid H \text{ is a half-space of a wall
separating } p \text{ and } gp\}.
\end{eqnarray*}
Thus,
$\vert Z_{gp}\, \triangle \, Z_p\vert = 2d(p,gp)$ and the properties
required of a zipper action follow.
\end{proof*}

{\footnotesize
\bibliographystyle{plain}
\addcontentsline{toc}{section}{\refname}
\bibliography{biblio}
}

{\footnotesize
\noindent
Department of Mathematics\\
Vanderbilt University\\
Nashville, TN 37240 USA\\
{bruce.hughes@vanderbilt.edu}
}

%%%%%%%%%%%%%%%%%%%%%%%%%%%%%%%%%%%%%%%%%%%%%%%%%%%%%%%%%%%%
\end{document}